\documentclass[final]{siamart251216}

\usepackage{amsmath,amssymb,mathtools}
\usepackage{algorithm}
\usepackage{algpseudocode}
\usepackage{booktabs}
\usepackage{enumitem}
\usepackage{graphicx}

\newsiamremark{remark}{Remark}

\newcommand{\R}{\mathbb{R}}

\newcommand{\Kcal}{\mathcal{K}}
\newcommand{\Omegacal}{\Omega}
\newcommand{\dist}{\operatorname{dist}}
\newcommand{\argmin}{\operatorname*{arg\,min}}
\newcommand{\lammin}{\lambda_{\min}}
\newcommand{\lammax}{\lambda_{\max}}

\title{Certificate-Coupled DCA for the Trust-Region Subproblem: Early Structural Escape and Progressive Global Certification}

\headers{Certificate-Coupled Difference-of-Convex Method}{N. T. Binh}

\author{
Binh T. Nguyen\thanks{
Facilty of Mathematics and Computer Science, University of Science, Vietnam National University Ho Chi Minh City,
Ho Chi Minh City, Vietnam. Email: ngtbinh@hcmus.edu.vn
}
}

\begin{document}

\maketitle

\begin{abstract}
We study the nonconvex trust-region subproblem by a certificate-coupled
difference-of-convex algorithm that incorporates spectral information before
exact first-order stationarity is attained.  The method maintains a persistent
randomized Krylov subspace together with the DCA iterates.  Near stationarity,
a sufficiently negative shifted Rayleigh quotient yields an explicit feasible
correction with a quantitative decrease in the objective value; otherwise,
completion of a prescribed Lanczos depth provides, with high probability, an
approximate positive-semidefiniteness certificate.  The same Krylov subspace also provides a certified upper bound on the largest eigenvalue for selecting the DC curvature parameter at subsequent accuracy stages.
We prove a uniform negative-curvature correction result over the trust region, covering the near-orthogonal hard-case regime, together with finite-termination and work bounds at fixed tolerances and a quantitative objective-gap certificate. Under a progressive accuracy schedule, the certified objective values converge to the global optimum and the iterates converge in distance to the global solution set, while the Krylov subspace is retained across corrections and accuracy stages. Numerical experiments illustrate the phenomena established by the analysis: negative-curvature corrections can occur before the corresponding post-convergence spectral corrections, the quantitative decrease bound is satisfied in all tested near-orthogonal instances, and retaining the Krylov subspace reduces the median number of spectral matrix-vector products by approximately a factor of $2.6$ relative to restarting the spectral process in the controlled comparison.
\end{abstract}

\begin{keywords}
difference-of-convex algorithm, trust-region subproblem, randomized Lanczos,
negative curvature, global optimality certificate, Krylov subspace method
\end{keywords}

\begin{MSCcodes}
90C20, 90C26, 65K05, 65F15
\end{MSCcodes}

\section{Introduction}
\label{sec:intro}

The trust-region subproblem (TRS) minimizes a possibly indefinite quadratic
function over a Euclidean ball.  Global optimality requires not only
first-order stationarity and complementarity, but also positive
semidefiniteness of an appropriately shifted Hessian; hence stationarity
alone is insufficient in the indefinite case.
Pham et al. \cite{PhamDinhLeThi1998} proposed a
difference-of-convex algorithm (DCA) for the TRS whose iterations reduce to
projections onto the Euclidean ball.  Their globalization procedure applies
a Lanczos spectral test after convergence of the DCA sequence and restarts
from a lower feasible point when negative curvature is detected.  The method
admits finitely many such restarts, and subsequent work established further
convergence and rate properties \cite{LeThiPhamDinhYen2012,Tuan2012,TuanYen2013}.

In this paper, spectral information is incorporated before exact first-order
stationarity is attained.  The method maintains a persistent randomized
Krylov subspace and, near stationarity, holds the primal iterate fixed while
the spectral approximation is refined.  A sufficiently negative shifted
Rayleigh direction yields a deterministic certificate of negative curvature
and an explicit feasible correction with a quantitative objective decrease.
Otherwise, completion of the prescribed Lanczos depth provides, with high
probability, an approximate positive-semidefiniteness certificate.  The
Krylov subspace is retained across corrections and accuracy stages and also
provides certified upper-spectral information for selecting the DC curvature
parameter.

Thus, unlike post-convergence spectral globalization, negative curvature can
modify the DCA sequence before exact stationarity while the accumulated
spectral information is retained.  This coupling yields finite
fixed-tolerance guarantees and, under progressive accuracy refinement,
convergence in objective value and in distance to the global solution set.

\paragraph{Contributions}
The main contributions of this paper can be given as follows.

First, we establish a uniform negative-curvature correction result for approximate stationary points of the TRS.  More precisely, if a feasible point has sufficiently small stationarity residual and admits a direction with sufficiently negative shifted Rayleigh quotient, then an explicit feasible correction can be constructed that yields a quantitatively bounded decrease in the objective value.  The result holds uniformly over the trust region and includes the near-orthogonal regime associated with the hard-case geometry.

Second, we develop a persistent bi-spectral certification procedure based on a single randomized Krylov subspace.  A sufficiently negative shifted Ritz value provides a deterministic certificate of violation of the positive-semidefiniteness condition and therefore permits an immediate correction.  Conversely, after a prescribed Lanczos depth, the absence of such a Ritz value yields, with high probability, an approximate positive-semidefiniteness certificate.  The same Krylov subspace provides certified upper-spectral information that is used to select the DC curvature parameter at subsequent stages.

Third, for prescribed tolerances $\varepsilon,\delta>0$, we prove finite termination of the resulting certificate-coupled DCA and derive explicit bounds on the numbers of DCA iterations, negative-curvature corrections, and spectral matrix-vector products.  The returned point satisfies
$$
S(x)\le\varepsilon,\qquad
\lambda_{\min}(A)+\lambda(x)\ge-\delta,
$$
and its objective value obeys the global error bound
$$
0\le q(x)-q^\star\le 2r\varepsilon+2r^2\delta.
$$
Fourth, we introduce a progressive certification scheme in which the tolerances are successively refined while the Krylov subspace is retained.  For the schedule $\varepsilon_m=r\delta_m$, the certified iterates satisfy
$$
0\le q(x_m)-q^\star\le4r^2\delta_m.
$$
Consequently, as $\delta_m\to0$, the objective values converge to $q^\star$ and
$$
\operatorname{dist}(x_m,X^\star)\to0.
$$
Moreover, the cumulative number of Krylov steps through a prescribed accuracy level is bounded by the order required at the final stage, up to fixed initialization overhead.

Finally, numerical experiments are used to examine the mechanisms predicted by the analysis.  In the controlled instances, certificate-coupled corrections occur before the corresponding post-convergence spectral corrections, the quantitative decrease bound is satisfied in all tested near-orthogonal cases, and retaining the Krylov subspace reduces the median spectral matrix-vector count by approximately a factor of $2.6$ relative to restarting the spectral process, without changing the corresponding DCA iteration and correction counts.

\subsection{Relation to prior work}
\label{subsec:related-work}

The closest antecedent is the DCA method for the TRS of Pham et al. \cite{PhamDinhLeThi1998}. There, DCA is run to stationarity, after which a Lanczos procedure is used to detect negative curvature and, when necessary, to restart the method from a lower feasible point; the resulting globalization requires only finitely many restarts. Later work established additional convergence and rate properties of DCA sequences for the TRS \cite{LeThiPhamDinhYen2012,Tuan2012,TuanYen2013}. The present method uses the same spectral characterization of global optimality, but employs spectral information before stationarity is attained and retains the Krylov subspace across negative-curvature corrections and accuracy stages.

More broadly, DCA and related DC methods have been extended through boosting and line-search strategies \cite{AbbaszadehpeivastiDeKlerkTaylor2026,AragonArtachoVuong2020}, Bregman formulations \cite{FaustFawziSaunderson2023,TakahashiFukudaTanaka2022}, and extrapolation \cite{PhanLeThi2023}. Recent analyses address nonasymptotic and worst-case rates \cite{AbbaszadehpeivastiDeKlerkZamani2024,RotaruPatrinosGlineur2025}, stronger stationarity notions \cite{LeThiHuynhPhamDinh2026}, and DC models within more general trust-region frameworks \cite{LeThiEtAl2014TR}. The present analysis instead concerns the incorporation of spectral certification into the DCA iteration for the TRS.

The TRS also admits non-DC approaches, including generalized-eigenvalue formulations \cite{AdachiIwataNakatsukasaTakeda2017}, globally convergent first-order methods \cite{BeckVaisbourd2018}, and large-scale Lanczos trust-region methods \cite{JiaWang2021,ZhangShenLi2017}; related variational properties are studied in \cite{JiangLi2022}. These methods provide alternative approaches to solving the TRS, whereas the present work concerns the incorporation of spectral certification within DCA.


\section{Problem Setup and DCA Dynamics}
\label{sec:setup}

We consider the trust-region subproblem
\begin{equation}
    q^\star
    :=
    \min_{x\in B_r} q(x),
    \qquad
    q(x)
    :=
    \frac12 x^\top A x+b^\top x,
    \qquad
    B_r:=\{x\in\R^n:\|x\|\le r\},
    \label{eq:trs}
\end{equation}
where $A=A^\top\in\R^{n\times n}$, $b\in\R^n$, and $r>0$.
We denote the global solution set by: 
$
    X^\star:=\argmin_{x\in B_r}q(x).
$
Throughout, we assume that a deterministic bound
\begin{equation}
    \|A\|_2\le M
    \label{eq:A-bound}
\end{equation}
is available.  We also fix a constant $\rho_{\min}>0$.
The gradient of the quadratic objective is denoted by
\begin{equation}
    g(x):=Ax+b.
    \label{eq:gradient}
\end{equation}
A feasible point $x^\star$ is globally optimal for \eqref{eq:trs} if and
only if there exists $\lambda^\star\ge0$ such that
\begin{equation}
    (A+\lambda^\star I)x^\star=-b,
    \qquad
    \lambda^\star(\|x^\star\|-r)=0,
    \qquad
    A+\lambda^\star I\succeq0.
    \label{eq:global-kkt}
\end{equation}
Thus, unlike first-order stationarity alone, global optimality contains a
spectral requirement on a shifted Hessian.  This difference drives the
certificate-coupled dynamics developed below.

\subsection{A difference-of-convex representation}
\label{subsec:dc-representation}

Whenever $\rho\ge\lammax(A)$, the objective in
\eqref{eq:trs}, including its trust-region indicator, admits the DC
representation
\begin{equation}
    q(x)+\chi_{B_r}(x)
    =
    G_\rho(x)-H_\rho(x),
    \label{eq:dc-decomposition}
\end{equation}
where
\begin{align}
    G_\rho(x)
    &:=
    \frac{\rho}{2}\|x\|^2+b^\top x+\chi_{B_r}(x),
    \label{eq:G-rho}\\
    H_\rho(x)
    &:=
    \frac12 x^\top(\rho I-A)x.
    \label{eq:H-rho}
\end{align}
Both $G_\rho$ and $H_\rho$ are convex.
The corresponding DCA update is
\begin{equation}
    D_\rho(x)
    :=
    P_{B_r}\left(x-\frac{1}{\rho}g(x)\right).
    \label{eq:dca-map}
\end{equation}
We allow $\rho$ to vary between certification stages, but every value used
by the algorithm will satisfy
\begin{equation}
    \rho\ge\max\{\rho_{\min},\lammax(A)\}.
    \label{eq:rho-admissible}
\end{equation}

\subsection{An exact stationarity residual}
\label{subsec:stationarity}

For $x\in B_r$, define
\begin{equation}
    \lambda(x)
    :=
    \begin{cases}
        0,
        & \|x\|<r,\\[1.5mm]
        \displaystyle
        \max\left\{
            0,
            -\frac{x^\top g(x)}{r^2}
        \right\},
        & \|x\|=r,
    \end{cases}
    \label{eq:lambda-x}
\end{equation}
and
\begin{equation}
    e(x):=g(x)+\lambda(x)x.
    \label{eq:e-x}
\end{equation}
We use the normal-cone residual
\begin{equation}
    S(x)
    :=
    \dist\bigl(0,g(x)+N_{B_r}(x)\bigr)
    \label{eq:S-def}
\end{equation}
as the first-order stationarity measure throughout the paper.

\begin{proposition}[Exact residual, DCA descent, and residual transfer]
\label{prop:dca-basic}
Let $x\in B_r$. Then,
\begin{equation}
    S(x)=\|e(x)\|.
    \label{eq:exact-residual}
\end{equation}
Moreover, let $\rho$ satisfy \eqref{eq:rho-admissible} and set
$z:=D_\rho(x)$. Thus,
\begin{equation}
    q(x)-q(z)
    \ge
    \left(\rho-\frac{\lammax(A)}{2}\right)\|z-x\|^2
    \ge
    \frac{\rho_{\min}}{2}\|z-x\|^2,
    \label{eq:dca-descent}
\end{equation}
and
\begin{equation}
    S(z)\le(\rho+M)\|z-x\|.
    \label{eq:residual-transfer}
\end{equation}
\end{proposition}
\begin{proof}
If $\|x\|<r$, then $N_{B_r}(x)=\{0\}$ and $\lambda(x)=0$, so
\eqref{eq:exact-residual} can be implied immediately.  If $\|x\|=r$, then
$
    N_{B_r}(x)=\{\mu x:\mu\ge0\}
$ and hence,
$
    S(x)=\min_{\mu\ge0}\|g(x)+\mu x\|.
$
Minimizing the squared norm in the scalar $\mu$ gives exactly
\eqref{eq:lambda-x} and proves \eqref{eq:exact-residual}. 
Let $z=D_\rho(x)$.  Projection optimality yields
\[
    \left[x-\frac1\rho g(x)-z\right]^\top(y-z)\le0
    \qquad (y\in B_r).
\]
Taking $y=x$ gives
\[
    g(x)^\top(z-x)\le-\rho\|z-x\|^2.
\]
Since
\[
    q(z)-q(x)
    =
    g(x)^\top(z-x)+\frac12(z-x)^\top A(z-x),
\]
we obtain
\[
    q(z)-q(x)
    \le
    -\left(\rho-\frac{\lammax(A)}2\right)\|z-x\|^2.
\]
Because $\rho\ge\lammax(A)$ and $\rho\ge\rho_{\min}$,
\[
    \rho-\frac{\lammax(A)}2
    \ge\frac{\rho}{2}
    \ge\frac{\rho_{\min}}2,
\]
which proves \eqref{eq:dca-descent}.
Finally, projection optimality also implies
\[
    n_z:=\rho(x-z)-g(x)\in N_{B_r}(z).
\]
Therefore,
\begin{align*}
    S(z)
    &\le\|g(z)+n_z\|
    =\|A(z-x)+\rho(x-z)\|\\
    &\le(\rho+\|A\|_2)\|z-x\|
    \le(\rho+M)\|z-x\|,
\end{align*}
which proves \eqref{eq:residual-transfer}.
\end{proof}


\section{Structural Certificates and Early Escape}
\label{sec:structural-escape}

Define
\begin{equation}
    B(x):=A+\lambda(x)I.
    \label{eq:B-x}
\end{equation}
A direction of sufficiently negative curvature for $B(x)$ certifies that a
nearly stationary point remains structurally incompatible with global
optimality.

\subsection{Uniform control of the shifted matrix}

\begin{lemma}[Uniform shifted-norm bound]
\label{lem:shifted-bound}
For every $x\in B_r$,
\begin{equation}
    0\le\lambda(x)\le M+\frac{\|b\|}{r}.
    \label{eq:lambda-bound}
\end{equation}
Consequently, with
$ \bar L:=2M+\frac{\|b\|}{r}$,
we have
\begin{equation}
    \|B(x)\|_2\le\bar L
    \qquad (x\in B_r).
    \label{eq:B-bound}
\end{equation}
\end{lemma}

\begin{proof}
The claim is immediate for $\|x\|<r$.  If $\|x\|=r$, then
\begin{align*}
    \lambda(x)
    &\le
    \frac{|x^\top(Ax+b)|}{r^2}
    \le
    \frac{\|A\|_2\|x\|^2+\|b\|\|x\|}{r^2}
    \le M+\frac{\|b\|}{r}.
\end{align*}
The bound on $\|B(x)\|_2$ follows from the triangle inequality.
\end{proof}

\subsection{A master perturbation identity}
\label{subsec:perturbation-identity}

\begin{lemma}[Feasible perturbation identity]
\label{lem:perturbation}
Let $x\in B_r$ and $d\in\R^n$.  If $\|x\|<r$, then
\begin{equation}
    q(x+d)-q(x)
    =
    e(x)^\top d+\frac12d^\top B(x)d.
    \label{eq:perturb-interior}
\end{equation}
If $\|x\|=r$ and $\|x+d\|=r$, then the same identity holds.
\end{lemma}

\begin{proof}
For every $d$,
\begin{equation}
    q(x+d)-q(x)
    =
    e(x)^\top d+\frac12d^\top B(x)d
    -
    \lambda(x)\left(x^\top d+\frac12\|d\|^2\right).
    \label{eq:master-general}
\end{equation}
The last term vanishes in the interior because $\lambda(x)=0$.
On the boundary, if $\|x+d\|=r=\|x\|$, then
$
    x^\top d+\frac12\|d\|^2=0.
$
\end{proof}

\subsection{Negative shifted curvature as an actionable certificate}

Let $v$ be any unit vector satisfying
\begin{equation}
    v^\top B(x)v\le-\kappa
    \qquad (\kappa>0).
    \label{eq:negative-witness}
\end{equation}
Define
\begin{equation}
    \alpha_\kappa
    :=
    \min\left\{
        \frac1{16},
        \frac{\kappa}{32\bar L}
    \right\}.
    \label{eq:alpha-kappa}
\end{equation}

\begin{theorem}[Uniform early structural escape]
\label{thm:early-escape}
Let $x\in B_r$, let $\|v\|=1$, and suppose
\[
    v^\top B(x)v\le-\kappa
\]
for some $\kappa>0$.  If
\begin{equation}
    S(x)
    \le
    \frac4{27}\kappa\alpha_\kappa r,
    \label{eq:escape-residual-condition}
\end{equation}
then one can explicitly construct $x^+\in B_r$ such that
\begin{equation}
    q(x)-q(x^+)
    \ge
    \frac4{27}\kappa\alpha_\kappa^2r^2.
    \label{eq:escape-decrease}
\end{equation}
\end{theorem}

\begin{proof}[Proof outline]
The complete geometric argument is deferred to
Appendix~\ref{app:escape-proof} and we formulate the common mechanism here.

Let
$
    h:=\sqrt{r^2-\|x\|^2}$, and
$
    c:=\frac{x^\top v}{r}.
$
The construction separates three exhaustive regimes:
(i) $h\ge\alpha_\kappa r$;
(ii) $h<\alpha_\kappa r$ and $|c|\ge\alpha_\kappa$;
and (iii) $h<\alpha_\kappa r$ and $|c|<\alpha_\kappa$.
In the first two cases, the farther intersection of the line generated by
$v$ with the trust-region sphere gives a feasible displacement $d$ with
$\|d\|\ge\alpha_\kappa r$ and sufficient negative shifted curvature.

In the third case, define
\[
    w:=v+\tau\frac{x}{r},
    \qquad
    \tau=2\alpha_\kappa\operatorname{sgn}(c),
\]
with either sign used when $c=0$.  Since
$\|B(x)\|_2\le\bar L$,
$
    w^\top B(x)w
    \le
    -\kappa+(2|\tau|+\tau^2)\bar L.
$
The choice \eqref{eq:alpha-kappa} implies
$
    w^\top B(x)w<-\frac34\kappa,
$
and
$
    \|w\|\le\frac98,
$
and hence,
\begin{equation}
    \frac{w^\top B(x)w}{\|w\|^2}
    \le-\frac{16}{27}\kappa.
    \label{eq:hard-curvature}
\end{equation}
The farther sphere intersection again gives $\|d\|\ge\alpha_\kappa r$.
Thus, all cases yield
\begin{equation}
    \frac12d^\top B(x)d
    \le-\frac8{27}\kappa\|d\|^2.
    \label{eq:uniform-quadratic-decrease}
\end{equation}
By Lemma~\ref{lem:perturbation},
\[
    q(x+d)-q(x)
    \le S(x)\|d\|-\frac8{27}\kappa\|d\|^2.
\]
Using \eqref{eq:escape-residual-condition} and
$\|d\|\ge\alpha_\kappa r$ gives
\[
    q(x+d)-q(x)
    \le-\frac4{27}\kappa\alpha_\kappa^2r^2.
\]
Taking $x^+=x+d$ proves the result.
\end{proof}



\section{Persistent Bi-Spectral Certification and the Algorithm}
\label{sec:bispectral}

The previous section turns negative-shifted curvature into an actionable
correction.  We now specify how such curvature information is obtained and
reused.  The key implementation principle is to maintain a single persistent
random-start Krylov process for the fixed matrix $A$.  Its lower spectral edge
drives within-stage structural decisions, whereas completed upper-edge
certificates calibrate the DC decomposition used by subsequent accuracy
stages.

\subsection{Persistent Krylov spaces and Ritz values}
\label{subsec:krylov}

Draw once
$
    q_1\sim\operatorname{Unif}(\mathbb S^{n-1}),
$
and define, for $j\ge1$,
\begin{equation}
    \Kcal^{(j)}
    :=
    \operatorname{span}
    \{q_1,Aq_1,\ldots,A^{j-1}q_1\}.
    \label{eq:krylov-space}
\end{equation}
Throughout the paper, $j$ denotes the \emph{Krylov order}.  Thus,
$\dim\Kcal^{(j)}\le j$, with strict inequality possible if the Lanczos
process reaches an invariant subspace.  If exact breakdown occurs before a
prescribed order, spectral refinement terminates at that invariant subspace;
under the random start used here, the resulting extremal Ritz values equal the
extremal eigenvalues of $A$ with probability one.  Such a breakdown can only
reduce the spectral work bounds below.

Define the extreme Ritz values
\begin{align}
    \theta_j^-
    &:=
    \min_{\substack{v\in\Kcal^{(j)}\\\|v\|=1}}
    v^\top Av,
    \label{eq:ritz-low}\\
    \theta_j^+
    &:=
    \max_{\substack{v\in\Kcal^{(j)}\\\|v\|=1}}
    v^\top Av.
    \label{eq:ritz-high}
\end{align}
Let $v_j^-$ and $v_j^+$ denote corresponding unit Ritz vectors whenever
needed.
Because
\begin{equation}
    \Kcal^{(j)}(A,q_1)=\Kcal^{(j)}(-A,q_1),
    \label{eq:krylov-sign}
\end{equation}
one persistent basis supports both spectral edges.

\subsection{Randomized two-edge spectral certification}
\label{subsec:lanczos-cert}

We use the following specialization of standard random-start Lanczos
bounds; see, e.g., \cite{KuczynskiWozniakowski1992,RoyerWright2018}.
The constants are written explicitly because they enter the certification
budget.

\begin{theorem}[Two-edge randomized Lanczos certificate]
\label{thm:lanczos-two-edge}
Let $A=A^\top$ satisfy $\|A\|_2\le M$.  Fix a spectral tolerance
$\delta>0$ and a failure budget $p\in(0,1)$.  Define
\begin{equation}
    J_\delta(p)
    :=
    \min\left\{
        n,\,
        1+
        \left\lceil
            \frac12
            \log\left(\frac{11n}{p^2}\right)
            \sqrt{\frac{2M}{\delta}}
        \right\rceil
    \right\}.
    \label{eq:J-delta-p}
\end{equation}
Then, there exists an event $\Omegacal_{\delta,p}$ with
\begin{equation}
    \Pr(\Omegacal_{\delta,p})\ge1-p
    \label{eq:omega-prob}
\end{equation}
such that, at $J:=J_\delta(p)$,
\begin{align}
    0
    &\le
    \theta_J^--\lammin(A)
    \le\frac{\delta}{4},
    \label{eq:lanczos-lower}\\
    0
    &\le
    \lammax(A)-\theta_J^+
    \le\frac{\delta}{4}.
    \label{eq:lanczos-upper}
\end{align}
\end{theorem}

\begin{proof}
The random-start Lanczos estimate used in
\cite{RoyerWright2018} implies the following one-edge statement.
For a symmetric matrix $H$ with $\|H\|_2\le M$, an absolute smallest-eigenvalue
error not exceeding $\epsilon/2$ is obtained with failure probability at most
$\pi$ after a sufficient Krylov depth bounded by
\[
    \min\left\{
        n,\,
        1+
        \left\lceil
            \frac12
            \log\left(\frac{2.75n}{\pi^2}\right)
            \sqrt{\frac{M}{\epsilon}}
        \right\rceil
    \right\}.
\]
Apply this statement to $H=A$ with
$
    \epsilon=\frac{\delta}{2},
$ and 
$
    \pi=\frac{p}{2}.
$
This gives \eqref{eq:lanczos-lower}.  Apply the same statement to $H=-A$,
using \eqref{eq:krylov-sign}; the minimum Ritz value of $-A$ is the negative
of the maximum Ritz value of $A$, yielding \eqref{eq:lanczos-upper}.
The union bound gives total failure probability at most $p$.
Substituting $\epsilon=\delta/2$ and $\pi=p/2$ yields
\eqref{eq:J-delta-p}.
\end{proof}

\begin{remark}
\label{rem:krylov-index}
Equation~\eqref{eq:J-delta-p} is interpreted as a sufficient Krylov
order.  In the absence of early exact breakdown, reaching order $J$ requires
$J+O(1)$ products with $A$, depending on the initialization convention.  If
exact breakdown occurs earlier, the required work is smaller.  We report
matrix-vector products separately in the computational study.
\end{remark}

\subsection{Deterministic negative-curvature detection and probabilistic certification}
\label{subsec:red-green}

Fix a spectral tolerance $\delta>0$ and define
\begin{equation}
\kappa_\delta:=\frac{3\delta}{4}.
\label{eq:kappa-delta}
\end{equation}
At a feasible point $x\in B_r$, let $(\theta_j^-,v_j^-)$ denote a
lower-edge Ritz pair obtained from the Krylov subspace of dimension $j$.
Since $v_j^-$ is normalized,
$$
    (v_j^-)^\top B(x)v_j^-
    =
    \theta_j^-+\lambda(x).
$$
\paragraph{Negative-curvature certificate}
At any Krylov order $j$, suppose that
\begin{equation}
\theta_j^-+\lambda(x)
\le -\frac{3\delta}{4}.
\label{eq:red-test}
\end{equation}
Then, $v_j^-$ is a directly verifiable direction of negative curvature for
$B(x)$, with
$$
    (v_j^-)^\top B(x)v_j^-
    \le -\kappa_\delta.
$$
Thus, whenever the stationarity-residual condition of
Theorem~\ref{thm:early-escape} is also satisfied, the theorem applies with
$\kappa=\kappa_\delta$ and yields an explicit feasible point with a
quantified decrease in the objective value.  This implication is
deterministic and does not rely on an accuracy guarantee for the Ritz
approximation.

\paragraph{Approximate positive-semidefiniteness certificate}
Failure of \eqref{eq:red-test} at an intermediate Krylov order does not
certify approximate positive semidefiniteness, since negative curvature may
remain undetected outside the current Krylov subspace.  Let
$J=J_\delta(p)$ be the prescribed randomized Lanczos depth.  Only after this
depth has been reached do we use the complementary condition
\begin{equation}
\theta_J^-+\lambda(x)>-\frac{3\delta}{4}
\label{eq:green-test}
\end{equation}
as a certificate.  On the event $\Omegacal_{\delta,p}$, the lower-edge
Lanczos bound gives
$$
    \lambda_{\min}(A)
    \ge \theta_J^- - \frac{\delta}{4},
$$
and therefore \eqref{eq:green-test} implies
$$
    \lambda_{\min}(A)+\lambda(x)>-\delta.
$$
Hence, the full-depth certificate establishes approximate positive
semidefiniteness of the shifted matrix $B(x)$ with the probability associated
with $\Omegacal_{\delta,p}$.
For brevity in the algorithmic description below, we refer to
\eqref{eq:red-test} as a \emph{RED} certificate and to the full-depth
condition \eqref{eq:green-test} as a \emph{GREEN} certificate.

\begin{proposition}[Validity of the spectral certificates]
\label{prop:red-green}
Let $x\in B_r$ and fix $\delta>0$ and $p\in(0,1)$. Then, the following
statements hold: 
\begin{enumerate}[label=(\alph*)]
\item If \eqref{eq:red-test} holds at some Krylov order $j$, then the
associated unit Ritz vector $v_j^-$ is a direction of negative curvature for
the shifted matrix $B(x)=A+\lambda(x)I$, with
\[
    (v_j^-)^\top B(x)v_j^-
    \le -\frac{3\delta}{4}.
\]
\item Let $J=J_\delta(p)$. On the event $\Omegacal_{\delta,p}$, if
\eqref{eq:green-test} holds, then
\begin{equation}
    \lambda_{\min}(A)+\lambda(x)>-\delta.
    \label{eq:green-valid}
\end{equation}
\end{enumerate}
\end{proposition}

\begin{proof}
For part~(a), since $v_j^-$ is a unit Ritz vector associated with
$\theta_j^-$,
$
    (v_j^-)^\top A v_j^-=\theta_j^-.
$
Therefore,
$
    (v_j^-)^\top B(x)v_j^-
    =
    \theta_j^-+\lambda(x)
    \le -\frac{3\delta}{4}.
$

For part~(b), the definition of $\Omegacal_{\delta,p}$ gives
$
    0\le
    \theta_J^- - \lambda_{\min}(A)
    \le \frac{\delta}{4},
$
and hence,
$
    \lambda_{\min}(A)
    \ge \theta_J^- - \frac{\delta}{4}.
$
Combining this inequality with
$
    \theta_J^-+\lambda(x)>-\frac{3\delta}{4}
$
yields
$$
    \lambda_{\min}(A)+\lambda(x)
    >
    -\frac{3\delta}{4}-\frac{\delta}{4}
    =-\delta,
$$
which proves \eqref{eq:green-valid}.
\end{proof}
\begin{remark}
\label{rem:asymmetric-cert}
The two spectral certificates are inherently asymmetric.  A direction with a
negative shifted Rayleigh quotient provides a deterministic certificate that
$B(x)$ is not positive semidefinite.  In contrast, the absence of negative
curvature from a finite Krylov subspace does not exclude its presence in the
orthogonal complement.  Consequently, approximate positive semidefiniteness
is asserted only after completion of the prescribed randomized Lanczos
budget, and the resulting certificate holds with the probability specified
by the event $\Omegacal_{\delta,p}$.
\end{remark}

\subsection{Certificate Lock}
\label{subsec:certificate-lock}

Near stationarity, the primal iterate is held fixed until the spectral
condition at that point has been resolved to the prescribed accuracy.  This
keeps the shifted matrix $B(x)$ unchanged while its extremal spectral
information is refined.

For fixed tolerances $\varepsilon,\delta>0$, define
\begin{equation}
    \alpha_\delta
    :=
    \min\left\{
        \frac1{16},
        \frac{3\delta}{128\bar L}
    \right\},
    \label{eq:alpha-delta}
\end{equation}
and
\begin{equation}
    \tau_{\varepsilon,\delta}
    :=
    \min\left\{
        \varepsilon,\,
        \frac{\delta\alpha_\delta r}{9}
    \right\}.
    \label{eq:tau-eps-delta}
\end{equation}
We also set
\begin{equation}
    \Delta_\delta
    :=
    \frac{\delta\alpha_\delta^2r^2}{9}.
    \label{eq:Delta-delta}
\end{equation}
These constants are chosen from Theorem~\ref{thm:early-escape} with
$\kappa=3\delta/4$.  Indeed,
\[
    \alpha_\delta=\alpha_{3\delta/4},
    \qquad
    \frac4{27}\left(\frac{3\delta}{4}\right)\alpha_\delta r
    =
    \frac{\delta\alpha_\delta r}{9}.
\]
Hence, $S(x)\le\tau_{\varepsilon,\delta}$ implies both
$S(x)\le\varepsilon$ and the residual condition required by
Theorem~\ref{thm:early-escape}.  Moreover, whenever
\[
    (v_j^-)^\top B(x)v_j^-
    \le-\frac{3\delta}{4},
\]
the decrease guaranteed by that theorem is
\[
    q(x)-q(x^+)
    \ge
    \frac4{27}\left(\frac{3\delta}{4}\right)
    \alpha_\delta^2r^2
    =
    \Delta_\delta.
\]

\begin{definition}[Certificate Lock]
\label{def:certificate-lock}
A feasible point $x$ enters \emph{Certificate Lock} when
\begin{equation}
    S(x)\le\tau_{\varepsilon,\delta}.
    \label{eq:lock-entry}
\end{equation}
While the lock is active, $x$ is held fixed and the persistent Krylov
subspace is enlarged until one of the following conditions occurs:
\begin{enumerate}
    \item the negative-curvature condition \eqref{eq:red-test} is satisfied,
    in which case Theorem~\ref{thm:early-escape} produces a feasible point
    $x^+$ satisfying
    \[
        q(x^+)\le q(x)-\Delta_\delta,
    \]
    after which the DCA iteration resumes from $x^+$; or
    \item the prescribed depth $J_\delta(p)$ is reached without
    \eqref{eq:red-test}, in which case, on $\Omegacal_{\delta,p}$,
    \[
        S(x)\le\varepsilon,
        \qquad
        \lambda_{\min}(A)+\lambda(x)>-\delta,
    \]
    and the current accuracy stage terminates.
\end{enumerate}
\end{definition}


\subsection{Stagewise upper-edge calibration}
\label{subsec:upper-calibration}

The upper spectral edge is used differently from the lower edge. We do not form an upper confidence bound from each intermediate Ritz value. Instead, upper-edge information is used to update the DC curvature parameter only after completion of a certified spectral stage.\\
Initialize
\begin{equation}
    U_{-1}:=M.
    \label{eq:U-minus-one}
\end{equation}
Suppose stage $m$ has completed with structural tolerance $\delta_m$ and
spectral depth $J_m:=J_{\delta_m}(p_m)$.  Define
\begin{equation}
    U_m
    :=
    \min\left\{
        U_{m-1},\,
        \theta_{J_m}^++\frac{\delta_m}{4}
    \right\}.
    \label{eq:U-update}
\end{equation}
The next stage uses
\begin{equation}
    \rho_{m+1}
    :=
    \max\{\rho_{\min},U_m\}.
    \label{eq:rho-stage}
\end{equation}
Also set
\begin{equation}
    \bar\rho:=\max\{\rho_{\min},M\}.
    \label{eq:rho-bar}
\end{equation}

\begin{proposition}[Validity of stagewise model calibration]
\label{prop:upper-calibration}
On the intersection of the successful upper-edge events through stage $m$,
\begin{equation}
    U_m\ge\lammax(A),
    \label{eq:U-safe}
\end{equation}
and consequently,
\begin{equation}
    \lammax(A)
    \le
    \rho_{m+1}
    \le
    \bar\rho.
    \label{eq:rho-safe}
\end{equation}
Moreover, $U_m\le U_{m-1}$.
\end{proposition}

\begin{proof}
The initialization $U_{-1}=M$ is safe because $\|A\|_2\le M$.  On the
successful upper-edge event at stage $m$,
\[
    \lammax(A)
    \le
    \theta_{J_m}^++\frac{\delta_m}{4}.
\]
If $U_{m-1}\ge\lammax(A)$, then the minimum in \eqref{eq:U-update} also
remains at least $\lammax(A)$.  Induction proves \eqref{eq:U-safe}.
The inequalities in \eqref{eq:rho-safe} follow from
\eqref{eq:rho-stage}, $U_m\le U_{-1}=M$, and the definition of
$\bar\rho$.
\end{proof}


\subsection{Progressive certificate-coupled DCA}
\label{subsec:algorithm}

The progressive method uses geometric structural tolerances
\begin{equation}
    \delta_m=\delta_0\beta^m,
    \qquad 0<\beta<1,
    \label{eq:delta-schedule}
\end{equation}
first-order tolerances
\begin{equation}
    \varepsilon_m=r\delta_m,
    \label{eq:epsilon-schedule}
\end{equation}
and failure allocation
\begin{equation}
    p_m=\frac{p}{2^{m+1}}.
    \label{eq:p-schedule}
\end{equation}
The resulting state consists of the primal point, the accuracy stage, the
persistent Krylov basis, and the latest certified upper spectral bound.

\begin{algorithm}[t]
\scalebox{0.8}{%
\caption{Progressive Bi-Spectral Certificate-Coupled DCA}
\label{alg:cc-dca}

\centering
\begin{minipage}{1.20\linewidth}
\begin{algorithmic}[1]
\Require
$A=A^\top$, $b$, $r$, $M\ge\|A\|_2$, $\rho_{\min}>0$,
$\delta_0>0$, $\beta\in(0,1)$, $p\in(0,1)$,
feasible $x^{(0)}$
\State Draw one random unit vector $q_1$ and initialize the persistent
       Lanczos/Krylov state with
       $\Kcal^{(1)}=\operatorname{span}\{q_1\}$
\State $j\gets1$ \Comment{current persistent Krylov order}
\State $U_{-1}\gets M$
\For{$m=0,1,2,\ldots$}
    \State $\delta_m\gets\delta_0\beta^m$,
           $\varepsilon_m\gets r\delta_m$,
           $p_m\gets p/2^{m+1}$
    \State compute $\alpha_{\delta_m}$,
           $\tau_m:=\tau_{\varepsilon_m,\delta_m}$, and
           $J_m:=J_{\delta_m}(p_m)$; if exact Lanczos breakdown has already
           occurred, set $j\gets J_m$
    \State $\rho_m\gets\max\{\rho_{\min},U_{m-1}\}$
    \State $x\gets x^{(m)}$
    \While{stage $m$ is unresolved}
        \If{$S(x)>\tau_m$}
            \State $x\gets D_{\rho_m}(x)$
        \Else
            \State activate \textsc{Certificate-Lock} at the fixed point $x$
            \While{current Krylov order $j<J_m$}
                \If{$\theta_j^-+\lambda(x)\le-3\delta_m/4$}
                    \State declare \textsc{Red}
                    \State construct the negative-curvature correction $x^+$
                           from Theorem~\ref{thm:early-escape}
                    \State $x\gets x^+$; release the lock
                    \State \textbf{break}
                \Else
                    \State extend the persistent Krylov process by one step;
                           set $j\gets j+1$, or $j\gets J_m$ upon exact breakdown
                \EndIf
            \EndWhile
            \If{the lock is active and $j=J_m$}
                \If{$\theta_{J_m}^-+\lambda(x)\le-3\delta_m/4$}
                    \State declare \textsc{Red}
                    \State construct the negative-curvature correction $x^+$
                           from Theorem~\ref{thm:early-escape}
                    \State $x\gets x^+$; release the lock
                \Else
                    \State declare \textsc{Green} and set $x_m\gets x$
                    \State terminate stage $m$
                \EndIf
            \EndIf
        \EndIf
    \EndWhile
    \State $U_m\gets
        \min\{U_{m-1},\theta_{J_m}^++\delta_m/4\}$
    \State $x^{(m+1)}\gets x_m$
\EndFor
\end{algorithmic}
\end{minipage}%
}
\end{algorithm}

\begin{remark}
\label{rem:persistent-state}
The Krylov basis is never restarted after a negative-curvature correction or at a new
accuracy stage.  This is valid because the matrix $A$ is fixed; only the
scalar shift $\lambda(x)I$ changes with the primal state.  The consequences
for finite-stage complexity and progressive spectral amortization are proved
in Sections~\ref{sec:fixed-stage} and~\ref{sec:progressive}.
\end{remark}


\section{Fixed-Tolerance Certification and Complexity}
\label{sec:fixed-stage}

We first analyze one accuracy stage with fixed tolerances
$(\varepsilon,\delta)$ and failure budget $p$.  Throughout this section,
let
\begin{equation}
    J:=J_\delta(p),
    \qquad
    \tau:=\tau_{\varepsilon,\delta},
    \qquad
    \Delta:=\Delta_\delta,
    \label{eq:fixed-stage-shorthand}
\end{equation}
and suppose that the DCA parameter used in the stage obeys
\begin{equation}
    \lammax(A)\le \rho\le\bar\rho.
    \label{eq:fixed-stage-rho}
\end{equation}
Let $x^{\rm in}$ be the stage input and denote its initial objective gap by
\begin{equation}
    \mathcal B
    :=
    q(x^{\rm in})-q^\star.
    \label{eq:stage-gap}
\end{equation}
Since all primal moves decrease the objective, $\mathcal B$ is an upper bound
on the objective budget available throughout the stage.

\subsection{Negative-curvature correction and DCA counts}
\label{subsec:fixed-counts}

Every negative-curvature correction occurs under Certificate Lock, so
$S(x)\le\tau$.  By \eqref{eq:tau-eps-delta},
\[
    S(x)
    \le
    \frac{\delta\alpha_\delta r}{9}
    =
    \frac4{27}\left(\frac{3\delta}{4}\right)
    \alpha_\delta r.
\]
The corresponding Ritz vector satisfies
\[
    (v_j^-)^\top B(x)v_j^-
    \le-\frac{3\delta}{4}.
\]
Thus, Theorem~\ref{thm:early-escape}, with
$\kappa=3\delta/4$, applies and every negative-curvature correction satisfies
\begin{equation}
    q(x)-q(x^+)
    \ge
    \Delta_\delta.
    \label{eq:red-decrease-fixed}
\end{equation}

\begin{proposition}[Number of structural corrections]
\label{prop:escape-count}
The number $N_E$ of negative-curvature corrections during a fixed-tolerance
stage satisfies
\begin{equation}
    N_E
    \le
    \left\lfloor
        \frac{\mathcal B}{\Delta_\delta}
    \right\rfloor.
    \label{eq:NE-bound}
\end{equation}
\end{proposition}

\begin{proof}
Every structural correction decreases $q$ by at least $\Delta_\delta$ by
\eqref{eq:red-decrease-fixed}, while $q$ is bounded below by $q^\star$.
After $N_E$ corrections,
\[
    N_E\Delta_\delta
    \le q(x^{\rm in})-q^\star
    =\mathcal B,
\]
which proves the result.
\end{proof}
We next count DCA steps.  Call a DCA step $x\mapsto z=D_\rho(x)$
\emph{far} if its successor still satisfies
\begin{equation}
    S(z)>\tau.
    \label{eq:far-step}
\end{equation}
By Proposition~\ref{prop:dca-basic},
\[
    S(z)\le(\rho+M)\|z-x\|
    \le(\bar\rho+M)\|z-x\|.
\]
Thus, every far step has
\begin{equation}
    \|z-x\|
    >
    \frac{\tau}{\bar\rho+M}.
    \label{eq:far-displacement}
\end{equation}
Combining this with the DCA decrease in
\eqref{eq:dca-descent} gives
\begin{equation}
    q(x)-q(z)
    >
    \frac{\rho_{\min}}{2}
    \frac{\tau^2}{(\bar\rho+M)^2}.
    \label{eq:far-decrease}
\end{equation}

\begin{proposition}[Number of far DCA steps]
\label{prop:far-count}
The number $N_{\rm far}$ of far DCA steps satisfies
\begin{equation}
    N_{\rm far}
    \le
    \left\lfloor
        \frac{
            2\mathcal B(\bar\rho+M)^2
        }{
            \rho_{\min}\tau_{\varepsilon,\delta}^2
        }
    \right\rfloor.
    \label{eq:Nfar-bound}
\end{equation}
\end{proposition}

\begin{proof}
Sum the uniform decrease \eqref{eq:far-decrease} over all far steps and use
$q\ge q^\star$.
\end{proof}

A remaining DCA step may have a successor satisfying $S(z)\le\tau$; such a
step enters a Certificate Lock.  Each lock either terminates the stage with a
full-depth certificate or ends with one negative-curvature correction.  Therefore,
the number of lock entries caused by DCA steps is at most one more than the number of negative-curvature
corrections.

\begin{proposition}[Certificate-Lock entries]
\label{prop:entry-count}
Let $N_{\rm enter}$ denote the number of DCA steps whose successor enters a
Certificate Lock.  Then,
\begin{equation}
    N_{\rm enter}\le N_E+1.
    \label{eq:Nenter-bound}
\end{equation}
Consequently,
\begin{equation}
    N_{\rm DCA}
    \le
    N_{\rm far}+N_E+1.
    \label{eq:NDCA-bound}
\end{equation}
\end{proposition}

\begin{proof}
Before the terminal certification lock, every completed lock must terminate
with a negative-curvature correction.  Hence, at most $N_E$ lock entries can be
followed by such a
correction and at most one additional entry can be terminal.  Every DCA step
is either far or a lock-entry step, proving \eqref{eq:NDCA-bound}.
\end{proof}

\subsection{Finite termination and certificate validity}
\label{subsec:fixed-termination}

The persistent Krylov state is refined only up to order $J$, unless exact Lanczos breakdown resolves the spectral problem earlier.
It is not restarted after a negative-curvature correction.  Once depth $J$ is
available, a later lock is resolved immediately by testing
\eqref{eq:red-test} and, if it fails, the full-depth condition
\eqref{eq:green-test}.

\begin{theorem}[Fixed-tolerance termination and certification]
\label{thm:fixed-stage}
Consider one stage of Algorithm~\ref{alg:cc-dca} with accuracy pair
$(\varepsilon,\delta)$, failure budget $p$, and DCA parameter satisfying
\eqref{eq:fixed-stage-rho}.  Let $\mathcal B$ be defined by
\eqref{eq:stage-gap}.  Then:
\begin{enumerate}
    \item the stage terminates after finitely many primal and spectral
    operations;
    \item its structural-correction and DCA counts satisfy
    \eqref{eq:NE-bound} and \eqref{eq:NDCA-bound};
    \item the stage requires at most $J_\delta(p)$ Krylov steps, and the
    number of newly generated products with $A$ is at most
    $J_\delta(p)+O(1)$;
    \item on the event $\Omegacal_{\delta,p}$ of
    Theorem~\ref{thm:lanczos-two-edge}, the stage output
    $x_{\varepsilon,\delta}$ satisfies
    \begin{align}
        S(x_{\varepsilon,\delta})
        &\le\varepsilon,
        \label{eq:fixed-output-residual}\\
        \lammin(A)+\lambda(x_{\varepsilon,\delta})
        &>-\delta.
        \label{eq:fixed-output-psd}
    \end{align}
\end{enumerate}
\end{theorem}

\begin{proof}
Propositions~\ref{prop:escape-count} and~\ref{prop:far-count} show that the
numbers of negative-curvature corrections and far DCA steps are finite.
Proposition~\ref{prop:entry-count} bounds the lock entries reached after DCA
steps.  In addition, there can be at most one initial lock entry, while any
lock entered immediately after a negative-curvature correction can be charged
to that correction.  Since $N_E<\infty$, the total number of Certificate-Lock
activations is therefore finite.

During a lock, the primal point is fixed and the persistent Krylov process is
refined only up to order $J$, or terminates earlier upon exact Lanczos
breakdown.  Hence every lock is resolved after finitely many spectral
operations, and the stage terminates.  Because the Krylov state persists
across all corrections, the spectral cost is bounded by the single prescribed
order $J$, up to the initialization convention of
Remark~\ref{rem:krylov-index}.

The stage terminates at a full-depth certification lock.  At that point
$S(x)\le\tau_{\varepsilon,\delta}\le\varepsilon$, proving
\eqref{eq:fixed-output-residual}.  On the event $\Omegacal_{\delta,p}$, Proposition~\ref{prop:red-green} gives
\eqref{eq:fixed-output-psd}.
\end{proof}

\begin{remark}
\label{rem:termination-correctness}
Finite stage termination is deterministic once the random starting vector is
fixed: at full Krylov depth the algorithm either detects a direction satisfying
\eqref{eq:red-test} or reaches the complementary full-depth condition
\eqref{eq:green-test}.  The event $\Omegacal_{\delta,p}$ is needed only to
certify that the latter implies the spectral guarantee
\eqref{eq:fixed-output-psd}.
\end{remark}

\section{Quantitative Global Certification}
\label{sec:global-cert}

The fixed-stage output combines a first-order residual with an approximate
positive-semidefiniteness certificate.  For the trust-region subproblem these
two quantities directly control the global objective error.

\begin{theorem}[Approximate certificates imply global objective accuracy]
\label{thm:approx-global}
Let $x\in B_r$ and let $\varepsilon,\delta\ge0$.  Suppose
\begin{equation}
    S(x)\le\varepsilon
    \label{eq:global-residual-assump}
\end{equation}
and
\begin{equation}
    A+\lambda(x)I\succeq-\delta I.
    \label{eq:global-psd-assump}
\end{equation}
Then
\begin{equation}
    0
    \le
    q(x)-q^\star
    \le
    2r\varepsilon+2r^2\delta.
    \label{eq:global-objective-bound}
\end{equation}
\end{theorem}

\begin{proof}
Fix any $y\in B_r$ and write $d:=y-x$.  From
\eqref{eq:master-general},
\begin{equation}
    q(y)-q(x)
    =
    e(x)^\top d
    +\frac12d^\top B(x)d
    -\lambda(x)
    \left(x^\top d+\frac12\|d\|^2\right).
    \label{eq:global-master}
\end{equation}
If $\|x\|<r$, then $\lambda(x)=0$.  If $\|x\|=r$, feasibility of $y$ gives
\[
    x^\top d+\frac12\|d\|^2
    =
    \frac12(\|y\|^2-r^2)
    \le0.
\]
Because $\lambda(x)\ge0$, the last term in \eqref{eq:global-master} is
therefore nonnegative in either case.  Using
$\|e(x)\|=S(x)\le\varepsilon$ and
$B(x)\succeq-\delta I$ yields
\begin{equation}
    q(y)-q(x)
    \ge
    -\varepsilon\|d\|
    -\frac{\delta}{2}\|d\|^2.
    \label{eq:global-lower}
\end{equation}
Take $y=x^\star\in X^\star$.  Since both $x$ and $x^\star$ lie in $B_r$,
$\|x-x^\star\|\le2r$.  Rearranging \eqref{eq:global-lower} gives
\[
    q(x)-q^\star
    \le
    2r\varepsilon+2r^2\delta.
\]
The lower bound in \eqref{eq:global-objective-bound} follows from global
optimality of $q^\star$.
\end{proof}

\begin{corollary}[Fixed-stage objective certificate]
\label{cor:fixed-global}
On $\Omegacal_{\delta,p}$, the output of
Theorem~\ref{thm:fixed-stage} satisfies
\begin{equation}
    q(x_{\varepsilon,\delta})-q^\star
    \le
    2r\varepsilon+2r^2\delta.
    \label{eq:fixed-global-bound}
\end{equation}
\end{corollary}

\begin{proof}
Combine \eqref{eq:fixed-output-residual} and
\eqref{eq:fixed-output-psd} with Theorem~\ref{thm:approx-global}.
\end{proof}


\section{Progressive Certification}
\label{sec:progressive}

We now combine the fixed-stage theory with the progressive schedule
\eqref{eq:delta-schedule}--\eqref{eq:p-schedule}.  Let $x_m$ denote the certified
output of stage $m$, and set
\begin{equation}
    J_m:=J_{\delta_m}(p_m).
    \label{eq:Jm}
\end{equation}
Define the all-stage successful spectral event
\begin{equation}
    \Omegacal_\infty
    :=
    \bigcap_{m=0}^{\infty}
    \Omegacal_{\delta_m,p_m}.
    \label{eq:Omega-infty}
\end{equation}
Since $\sum_{m\ge0}p_m=p$, the union bound gives
\begin{equation}
    \Pr(\Omegacal_\infty)\ge1-p.
    \label{eq:Omega-infty-prob}
\end{equation}
No independence between stages is required.

\subsection{Global convergence of the certified sequence}

\begin{theorem}[Progressive global certification]
\label{thm:progressive}
Run Algorithm~\ref{alg:cc-dca} with
\[
    \delta_m=\delta_0\beta^m,
    \qquad
    \varepsilon_m=r\delta_m,
    \qquad
    p_m=\frac{p}{2^{m+1}},
\]
where $0<\beta<1$ and $0<p<1$.  On the event
$\Omegacal_\infty$, every stage is correctly certified and
\begin{align}
    S(x_m)
    &\le r\delta_m,
    \label{eq:progressive-residual}\\
    \lammin(A)+\lambda(x_m)
    &>-\delta_m,
    \label{eq:progressive-psd}\\
    0\le q(x_m)-q^\star
    &\le4r^2\delta_m.
    \label{eq:progressive-objective}
\end{align}
Consequently,
\begin{equation}
    q(x_m)\to q^\star,
    \label{eq:objective-convergence}
\end{equation}
and
\begin{equation}
    \dist(x_m,X^\star)\to0.
    \label{eq:set-convergence}
\end{equation}
If $X^\star=\{x^\star\}$ is a singleton, then
\begin{equation}
    x_m\to x^\star.
    \label{eq:unique-convergence}
\end{equation}
\end{theorem}

\begin{proof}
At stage $0$, $U_{-1}=M$ gives
$\rho_0=\max\{\rho_{\min},M\}=\bar\rho\ge\lambda_{\max}(A)$.  On
$\Omegacal_\infty$, Proposition~\ref{prop:upper-calibration} gives
$\lambda_{\max}(A)\le\rho_m\le\bar\rho$ for every $m\ge1$.  Hence
Theorem~\ref{thm:fixed-stage} applies at every stage, giving
\eqref{eq:progressive-residual} and \eqref{eq:progressive-psd}.  Theorem~\ref{thm:approx-global} and
$\varepsilon_m=r\delta_m$ then imply
\[
    q(x_m)-q^\star
    \le
    2r(r\delta_m)+2r^2\delta_m
    =4r^2\delta_m,
\]
which proves \eqref{eq:progressive-objective}.  Since
$\delta_m\to0$, \eqref{eq:objective-convergence} follows.

It remains to prove convergence to the solution set.  Suppose, to the
contrary, that \eqref{eq:set-convergence} fails.  Then, there exist
$\eta>0$ and a subsequence $\{x_{m_k}\}$ such that
\[
    \dist(x_{m_k},X^\star)\ge\eta
    \qquad\text{for all }k.
\]
The trust-region ball $B_r$ is compact, so a further subsequence converges to
some $\bar x\in B_r$.  By continuity of $q$ and
\eqref{eq:objective-convergence},
$
    q(\bar x)=q^\star,
$
so $\bar x\in X^\star$, contradicting the distance lower bound.  

Hence,
\eqref{eq:set-convergence} holds.  If the solution set is a singleton,
set convergence is equivalent to \eqref{eq:unique-convergence}.
\end{proof}

\subsection{Persistent spectral amortization}
\label{subsec:spectral-amortization}

Because $\delta_m$ decreases and $p_m$ decreases, the sufficient Krylov
depths $J_m$ in \eqref{eq:J-delta-p} are nondecreasing.  The basis is never
restarted, so stage $m$ extends the spectral state only from its previous
depth to $J_m$.

\begin{proposition}[Spectral amortization]
\label{prop:spectral-amortization}
Through the end of stage $M$, Algorithm~\ref{alg:cc-dca} requires at most
$J_M$ Krylov steps.  In the absence of early exact Lanczos breakdown, the
successive extensions telescope as
\begin{equation}
    J_0+(J_1-J_0)+\cdots+(J_M-J_{M-1})=J_M.
    \label{eq:krylov-telescope}
\end{equation}
Early exact breakdown can only reduce the required spectral work.  Under the
matrix-vector counting convention of Remark~\ref{rem:krylov-index},
\begin{equation}
    N_{Av}^{0:M}
    \le
    J_M+O(1).
    \label{eq:matvec-amortized}
\end{equation}
\end{proposition}

\begin{proof}
The Krylov spaces are nested by construction and the same random starting
vector and spectral state are retained throughout all stages and structural
corrections.  Since $J_m$ is nondecreasing, in the absence of early exact
breakdown only the incremental orders $J_m-J_{m-1}$ need to be generated at
stage $m$.  Summing these increments gives \eqref{eq:krylov-telescope}.
If exact breakdown occurs earlier, no further Krylov steps are required.
The matrix-vector bound \eqref{eq:matvec-amortized} then follows from
Remark~\ref{rem:krylov-index}.
\end{proof}

\subsection{Warm-start stage complexity}
\label{subsec:warm-complexity}

The progressive schedule also sharpens the stagewise primal complexity
because each stage starts from the already certified output of the previous
stage.  We record explicit small-$\delta$ bounds.

Suppose
\begin{equation}
    \delta\le\frac{8\bar L}{3}.
    \label{eq:small-delta-regime}
\end{equation}
Then, the second branch in \eqref{eq:alpha-delta} is active, so
\begin{equation}
    \alpha_\delta
    =
    \frac{3\delta}{128\bar L},
    \qquad
    \tau_{r\delta,\delta}
    =
    \frac{\delta^2r}{384\bar L},
    \qquad
    \Delta_\delta
    =
    \frac{\delta^3r^2}{16384\bar L^2}.
    \label{eq:small-delta-values}
\end{equation}
For every $m\ge1$, \eqref{eq:progressive-objective} at stage $m-1$ gives
\begin{equation}
    \mathcal B_m^{\rm in}
    :=
    q(x_{m-1})-q^\star
    \le
    4r^2\delta_{m-1}
    =
    \frac{4r^2}{\beta}\delta_m.
    \label{eq:warm-budget}
\end{equation}

\begin{corollary}[Progressive warm-start complexity]
\label{cor:warm-complexity}
Assume stage $m\ge1$ satisfies
$\delta_m\le8\bar L/3$.  On the all-stage successful event, its correction
and far-DCA counts satisfy
\begin{align}
    N_{E,m}
    &\le
    \frac{65536\,\bar L^2}{\beta\delta_m^2},
    \label{eq:warm-NE}\\
    N_{{\rm far},m}
    &\le
    \frac{
        1179648\,\bar L^2(\bar\rho+M)^2
    }{
        \beta\rho_{\min}\delta_m^3
    }.
    \label{eq:warm-Nfar}
\end{align}
Thus, the progressive warm-start stage bounds scale as
\begin{equation}
    N_{E,m}=O(\delta_m^{-2}),
    \qquad
    N_{{\rm far},m}=O(\delta_m^{-3}).
    \label{eq:warm-rates}
\end{equation}
By contrast, applying the fixed-stage bounds with an accuracy-independent
initial objective budget gives the corresponding cold-budget powers
$O(\delta^{-3})$ and $O(\delta^{-4})$.
\end{corollary}

\begin{proof}
Substitute \eqref{eq:small-delta-values} and
\eqref{eq:warm-budget} into \eqref{eq:NE-bound} and
\eqref{eq:Nfar-bound}.  For the correction count,
\[
    \frac{\mathcal B_m^{\rm in}}{\Delta_{\delta_m}}
    \le
    \frac{4r^2\delta_m/\beta}
         {\delta_m^3r^2/(16384\bar L^2)}
    =
    \frac{65536\bar L^2}{\beta\delta_m^2}.
\]
For the far-step count,
\begin{align*}
    \frac{
        2\mathcal B_m^{\rm in}(\bar\rho+M)^2
    }{
        \rho_{\min}\tau_{r\delta_m,\delta_m}^2
    }
    &\le
    \frac{
        8r^2\delta_m(\bar\rho+M)^2/\beta
    }{
        \rho_{\min}\delta_m^4r^2/(147456\bar L^2)
    }\\
    &=
    \frac{
        1179648\bar L^2(\bar\rho+M)^2
    }{
        \beta\rho_{\min}\delta_m^3
    }.
\end{align*}
The cold-budget exponents follow directly from
$\Delta_\delta=\Theta(\delta^3)$ and
$\tau_{r\delta,\delta}=\Theta(\delta^2)$ when the initial objective gap is
regarded as independent of $\delta$.
\end{proof}

For geometric tolerances, cumulative stagewise polynomial bounds remain of
the order of the finest stage.  Indeed, for every $s>0$,
\begin{equation}
    \sum_{m=1}^{M}\delta_m^{-s}
    \le
    \frac{\delta_M^{-s}}{1-\beta^s}.
    \label{eq:geometric-sum}
\end{equation}

\subsection{Main results}
\label{subsec:main-theorem}

The preceding results can be summarized as follows.

\begin{theorem}[Progressive certificate-coupled DCA]
\label{thm:main}
Let $A=A^\top\in\R^{n\times n}$, $r>0$, and $\|A\|_2\le M$.
Fix $\rho_{\min}>0$, $0<\beta<1$, $0<p<1$, and
$0<\delta_0\le M$.  Run Algorithm~\ref{alg:cc-dca} with
\[
    \delta_m=\delta_0\beta^m,
    \qquad
    \varepsilon_m=r\delta_m,
    \qquad
    p_m=\frac{p}{2^{m+1}}.
\]
Then, with probability at least $1-p$, the following statements hold
simultaneously for all stages $m$:
\begin{enumerate}
    \item The DC parameter is valid, $\rho_m\ge\lambda_{\max}(A)$, and
    every stage terminates after finitely many DCA updates,
    negative-curvature corrections, and spectral refinements.

    \item Whenever
    \[
        \theta_j^-+\lambda(x)\le-\frac{3\delta_m}{4}
    \]
    under Certificate Lock, an explicit feasible correction $x^+$ satisfies
    \[
        q(x)-q(x^+)
        \ge
        \frac{\delta_m\alpha_{\delta_m}^2r^2}{9}.
    \]
    This implication is deterministic.  The stage output satisfies
    \[
        S(x_m)\le r\delta_m,
        \qquad
        \lambda_{\min}(A)+\lambda(x_m)>-\delta_m,
        \qquad
        0\le q(x_m)-q^\star\le4r^2\delta_m.
    \]

    \item Consequently,
    \[
        q(x_m)\to q^\star,
        \qquad
        \operatorname{dist}(x_m,X^\star)\to0.
    \]

    \item Through any stage $M$, at most $J_M$ Krylov steps are required,
    and the number of products with $A$ used to construct the persistent
    Krylov subspace is at most $J_M+O(1)$.  Moreover, once
    $\delta_m\le8\bar L/3$,
    \[
        N_{E,m}=O(\delta_m^{-2}),
        \qquad
        N_{{\rm far},m}=O(\delta_m^{-3}).
    \]
\end{enumerate}
The probability statement is required only for the full-depth approximate
positive-semidefiniteness certificate and the certified upper-spectral
bound.
\end{theorem}
\begin{proof}
At stage $0$, $U_{-1}=M$ gives
$\rho_0=\bar\rho\ge\lambda_{\max}(A)$; for $m\ge1$, the validity of the stagewise
DC parameters follows from Proposition~\ref{prop:upper-calibration}.  Finite
stage termination follows from Theorem~\ref{thm:fixed-stage}.  The negative-curvature decrease in
item~2 is Theorem~\ref{thm:early-escape} with
$\kappa=3\delta_m/4$ and the lock threshold
\eqref{eq:tau-eps-delta}; the remaining stagewise certificates and the
convergence statement in item~3 follow from
Theorem~\ref{thm:progressive}.  Finally, the Krylov work bound and the
warm-start estimates in item~4 follow from
Proposition~\ref{prop:spectral-amortization} and
Corollary~\ref{cor:warm-complexity}, respectively.

The deterministic character of the negative-curvature implication and the
probabilistic character of the full-depth spectral certificate follow from
Proposition~\ref{prop:red-green} and
Remark~\ref{rem:asymmetric-cert}.  Applying these results on the common
all-stage Lanczos event completes the proof.
\end{proof}


\section{Numerical Experiments}
\label{sec:experiments}

We report numerical experiments examining the principal features of the
proposed method.  Each experiment uses 20 matched random seeds.  Computational
work is measured by the number $N_{Av}$ of matrix-vector products with $A$,
the number $N_{\rm DCA}$ of DCA updates, and the number $N_E$ of accepted
negative-curvature corrections.  Experimental configurations, randomization,
and reference-solution tolerances are given in
Appendix~\ref{app:experimental-protocol}.  The study comprises 1,300 runs.

\paragraph{Methods and comparisons.}
We use four variants: plain DCA; post-convergence spectral globalization,
which runs DCA to the prescribed first-order tolerance before a cold
lower-edge spectral check and, when necessary, a two-dimensional feasible
circle correction; certificate-coupled DCA with Krylov restarts after
corrections or accuracy stages; and persistent bi-spectral CC-DCA
(Algorithm~\ref{alg:cc-dca}).  E1 compares all four variants, E3 compares the
two certificate-coupled variants, and E4 compares progressive certification
with a direct one-stage run at the final tolerance.  E2 examines the
quantitative decrease estimate, while E5 examines spectral sensitivity and
generic random instances.  Global reference solutions are computed from the
TRS KKT/secular system and checked for feasibility, stationarity,
complementarity, and shifted positive semidefiniteness.

\subsection{Experiment 1 (E1): acting before post-convergence certification}
\label{subsec:e1-results}

The controlled family has $\lambda_1<\lambda_2<0$, $b=0$, with initialization
in an invariant subspace excluding the lowest eigendirection; plain DCA thus
approaches a nonglobal first-order stationary point.  Across 80 matched
dimension--seed pairs, the median DCA counts before the first
negative-curvature correction are $40.0$ for the post-convergence method and
$17.5$ for both certificate-coupled variants.  The latter correction occurs
$22$--$23$ updates earlier on every matched instance, with median difference
$22.5$.  Plain DCA terminates with median objective gap $1.000$.

The cold and persistent certificate-coupled variants have identical
first-correction, $N_{\rm DCA}$, and $N_E$ counts on all 80 matched instances,
with median spectral counts $35.0$ and $31.5$, respectively.  Both attain
median final objective gap $8.37\times10^{-10}$.  Thus E1 separates the
earlier correction due to certificate coupling from the reduction in repeated
spectral work due to Krylov persistence; see
Figure~\ref{fig:mechanism-ablation}(a).

\begin{figure}[t]
\centering
\begin{minipage}{.48\linewidth}
\centering
\includegraphics[width=0.8\linewidth]{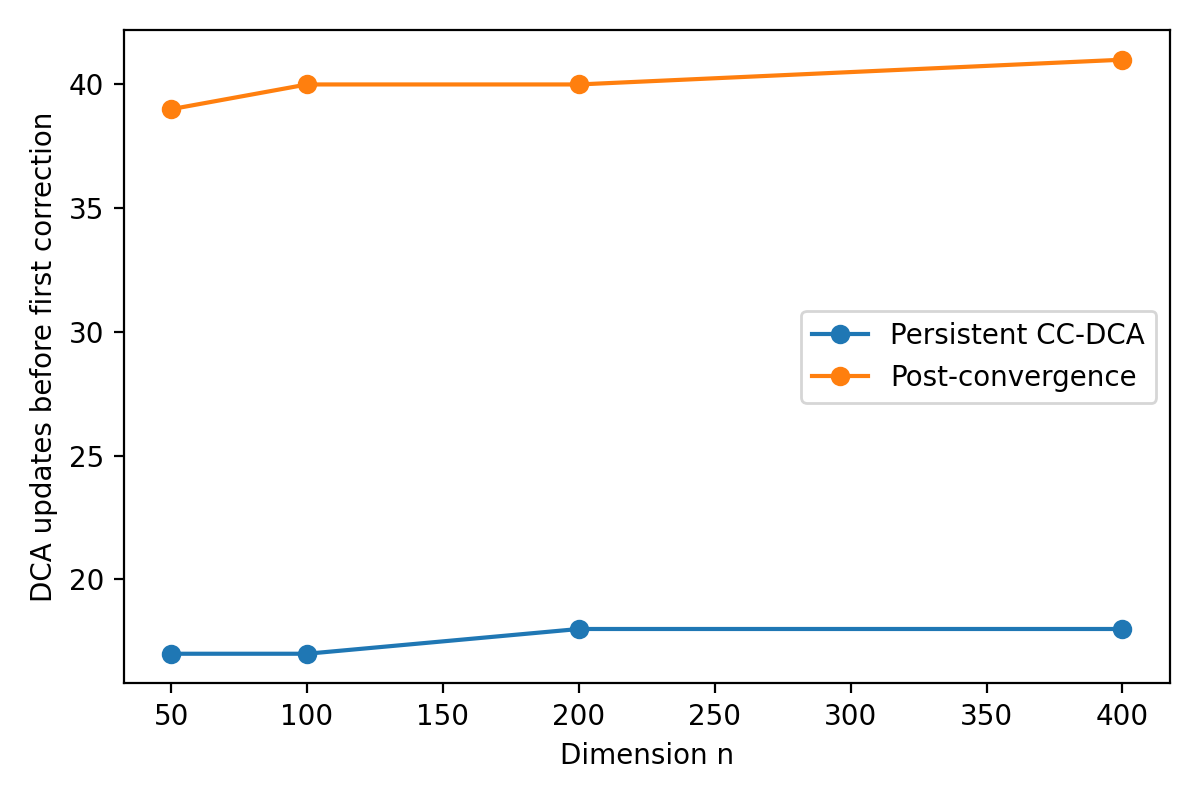}\\[-1mm]
{\small (a) E1: DCA updates before first correction.}
\end{minipage}\hfill
\begin{minipage}{.48\linewidth}
\centering
\includegraphics[width=0.8\linewidth]{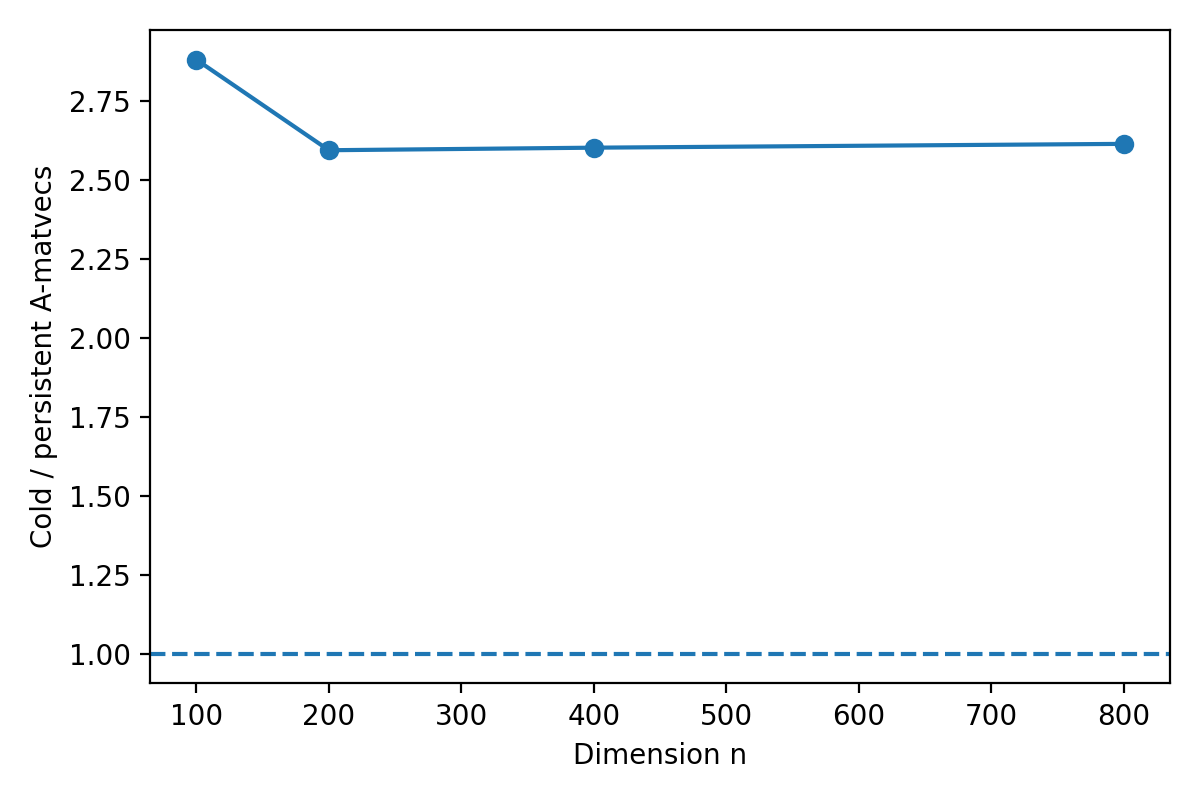}\\[-1mm]
{\small (b) E3: cold/persistent spectral-work ratio.}
\end{minipage}
\caption{E1 and E3 comparisons.  (a) The certificate-coupled correction
precedes the post-convergence correction.  (b) Krylov persistence reduces
repeated spectral work; values above one indicate a reduction relative to
restarting.}
\label{fig:mechanism-ablation}
\end{figure}

\subsection{Experiment 2 (E2): near-orthogonal negative-curvature correction}
\label{subsec:e2-results}

We test the quantitative decrease estimate of
Theorem~\ref{thm:early-escape} for $n=200$ and six alignment levels from exact
orthogonality to $9\times10^{-4}$.  All 120 instances satisfy the theorem
hypotheses.  For
\[
 R_{\rm dec}:=\frac{q(x)-q(x^+)}{\Delta_\delta},
\]
the minimum, median, and maximum are $53.983$, $54.221$, and $57.868$.
Thus $R_{\rm dec}\ge1$ throughout, including at exact orthogonality; the
uniform decrease estimate is conservative over the tested configurations.

\subsection{Experiment 3 (E3): Krylov persistence}
\label{subsec:e3-results}

We compare the cold and persistent certificate-coupled variants using matched
initializations, tolerance schedules, and random starts.  Across 80 runs with
$n\in\{100,200,400,800\}$,
\[
 R_{Av}:=\frac{N_{Av}^{\rm cold}}{N_{Av}^{\rm persistent}}
\]
has median $2.612$, interquartile range $[2.594,2.692]$, and range
$[2.593,2.920]$.  The variants have identical $N_{\rm DCA}$ and $N_E$ on all
80 matched runs.  Hence Krylov persistence reduces spectral matrix-vector
products in this comparison without changing the corresponding DCA-update or
negative-curvature-correction counts; see
Figure~\ref{fig:mechanism-ablation}(b).

\subsection{Experiment 4 (E4): progressive certification}
\label{subsec:e4-results}

We compare five progressive stages $\delta_m=2^{-m}$, $m=0,\ldots,4$, with a
direct one-stage run at $\delta_{\rm final}=0.0625$.  Progressive certification
saves a median of 12 DCA updates but uses a median of 35 additional spectral
matrix-vector products.  Its purpose is therefore not necessarily to reduce
total spectral work, but to produce increasingly accurate certified outputs
while retaining primal and spectral information across stages.  Every final
output satisfies
\[
 q(x)-q^\star\le4r^2\delta_{\rm final},
\]
and the largest observed ratio of objective gap to certificate radius is
$1.04\times10^{-12}$.

\subsection{Experiment 5 (E5): spectral sensitivity and generic robustness}
\label{subsec:e5-results}

For controlled lower spectral gaps $0.1,0.25,0.5,1,$ and $2$, the median
$N_{\rm DCA}$ values are, respectively,
\[
 1655,\quad669,\quad340,\quad160.5,\quad83.
\]
Varying bottom-eigenvalue multiplicity over $\{1,2,4,8,16\}$ or the number of
additional negative directions over $\{0,2,5,10,20\}$ produces only modest
changes in median DCA work.  As $n$ increases from $100$ to $800$, median
$N_{Av}$ increases from $100$ to $132$, while median $N_{\rm DCA}$ changes
only mildly.  Figure~\ref{fig:e5-gap} summarizes the lower-gap dependence.
\begin{figure}[t]
\centering
\includegraphics[width=.60\linewidth]{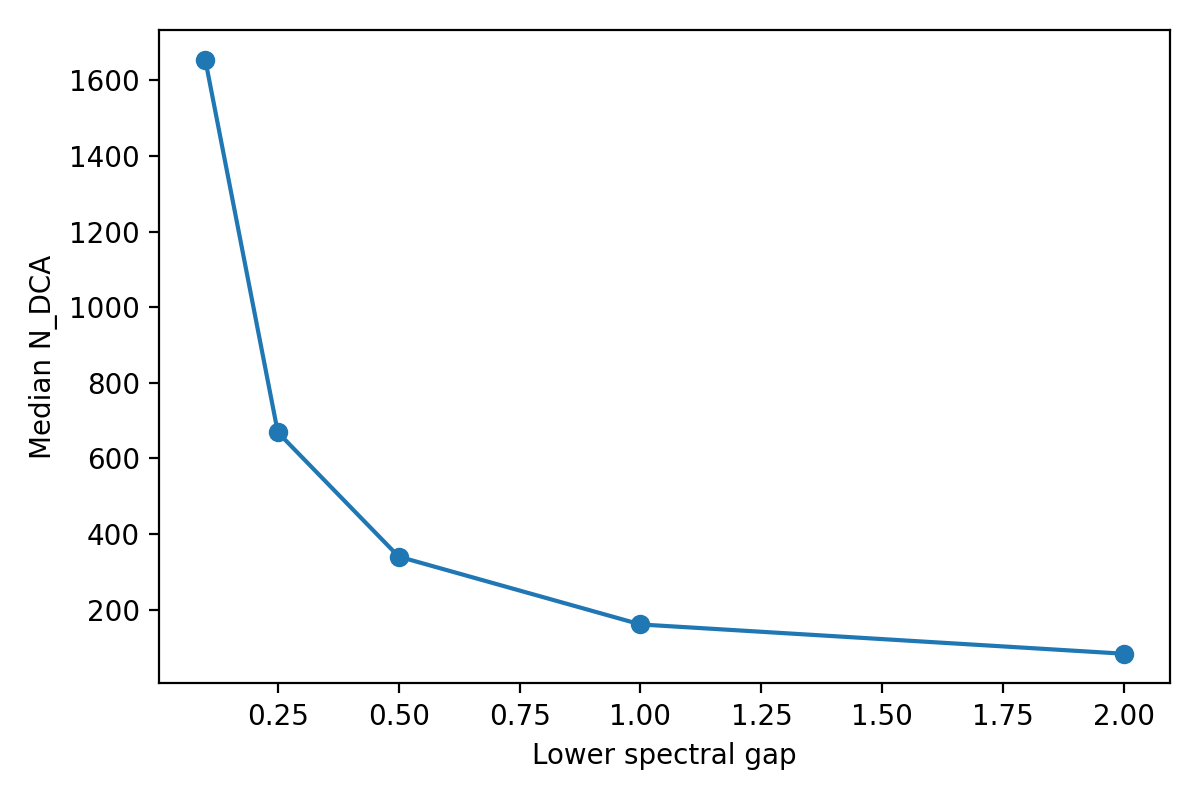}
\caption{E5: median DCA work versus the controlled lower spectral gap.}
\label{fig:e5-gap}
\end{figure}
In the experiment, we additionally consider 240 random spectral TRS instances with
$n\in\{100,200,400\}$, negative-eigenvalue fractions $\{0.1,0.25\}$, and
linear-term scales $\{0.1,0.5\}$.  All 240 runs terminate with the prescribed
full-depth spectral certificate.  The largest objective gap is $0.07498$, or
$0.300$ of the final certificate radius.  The largest reference-solution
stationarity residual is $1.10\times10^{-14}$, with no shifted
positive-semidefiniteness violation beyond the prescribed tolerance.

Taken together, the experiments show earlier negative-curvature correction
on E1, validity of the quantitative decrease estimate in the tested
near-orthogonal regime, reduced repeated spectral work from Krylov persistence,
and the dependence of computational work on spectral structure.

\section{Discussion and Conclusions}
\label{sec:discussion}

The classical DCA approach to the trust-region subproblem combines projected
DCA iterations with a spectral globalization step applied after stationarity
has been reached.  In this work, spectral information is incorporated earlier
and retained throughout the computation.  A sufficiently negative shifted
Rayleigh quotient yields a feasible negative-curvature correction with a
quantified decrease in the objective value, whereas completion of the
prescribed randomized Lanczos depth yields, with high probability, an
approximate positive-semidefiniteness certificate.  The same Krylov subspace
also provides an upper bound on the largest eigenvalue for selecting the DC
curvature parameter at subsequent accuracy stages.  Separating spectral
refinement from DCA updates near stationarity leads to additive bounds on
first-order and Krylov work.

The numerical experiments are consistent with these conclusions.  On the
controlled instances, negative-curvature corrections occur before the
corresponding post-convergence spectral corrections, and the quantitative
decrease estimate remains valid in the tested near-orthogonal regime.
Retaining the Krylov subspace reduces repeated spectral matrix-vector products
in the controlled persistence comparison.  The progressive experiments
further illustrate the use of the retained spectral information across a
sequence of increasingly accurate global certificates.

The analysis in this paper is restricted to the Euclidean trust-region
subproblem.  Extensions to other DC models admitting suitable spectral
certificates, removal of the a priori spectral-norm bound, and less
conservative use of intermediate upper spectral information remain topics for
future investigation.  The results show that, for the trust-region
subproblem, spectral information need not be confined to a globalization test
performed after stationarity: it can be incorporated within the DCA iteration
while preserving explicit finite-tolerance work bounds and progressive global
optimality guarantees.

\section*{Declarations}
AI-assisted tools were used during manuscript preparation to support
exploratory proof checking, consistency checks, computational implementation
and code refactoring, and editorial refinement. All final mathematical claims,
proofs, citations, computational results, and code outputs were independently
reviewed and verified by the author, who takes full responsibility for the
content of the article. Source code, experimental configurations, and
reproducibility data are publicly available at
\url{https://github.com/aisia-research-lab/certificate-coupled-dca}.

\bibliographystyle{siamplain}
\bibliography{cc_dca_references}

\appendix
\section{Experimental Protocol and Reproducibility}
\label{app:experimental-protocol}

The experimental design was fixed before the main runs.  Seeds are
$0,\ldots,19$, with deterministic offsets for the Lanczos starts.  E1 uses
$n\in\{50,100,200,400\}$ and $\delta=0.5$; E2 uses $n=200$ and six
alignment levels from $0$ to $9\times10^{-4}$; and E3--E4 use
$n\in\{100,200,400,800\}$ with $\delta_m=2^{-m}$,
$m=0,\ldots,4$.  E5 varies the lower spectral gap, bottom-eigenvalue
multiplicity, number of additional negative directions, and dimension; its
generic random family additionally varies the fraction of negative
eigenvalues and the linear-term scale.  Controlled instances use seeded
matrices $A=Q\Lambda Q^\top$.  Reference solutions are computed independently
from the classical TRS KKT/secular system and checked for feasibility,
stationarity, complementarity, and shifted positive semidefiniteness.

Each run records $N_{\rm DCA}$, $N_E$, $N_{Av}$, random seeds,
objective and residual diagnostics, termination status, and runtime.
The source code, experimental configuration, environment metadata, raw
results, checkpoints, and summaries are available in the public
reproducibility repository at 
\url{https://github.com/aisia-research-lab/certificate-coupled-dca}.
Offline eigendecompositions used only for instance construction or
verification are excluded from $N_{Av}$.

\section{Geometric Proof of the Early Structural Escape Theorem}
\label{app:escape-proof}

We give the complete construction used in
Theorem~\ref{thm:early-escape}.  Fix $x\in B_r$ and a unit vector $v$ such
that
\begin{equation}
    v^\top B(x)v\le-\kappa,
    \qquad \kappa>0,
    \label{eq:app-neg-curv}
\end{equation}
and abbreviate
\begin{equation}
    \alpha:=\alpha_\kappa
    =\min\left\{\frac1{16},\frac{\kappa}{32\bar L}\right\},
    \qquad
    h:=\sqrt{r^2-\|x\|^2},
    \qquad
    c:=\frac{x^\top v}{r}.
    \label{eq:app-alpha-h-c}
\end{equation}
The proof repeatedly uses an elementary formula for the farther intersection
of a line with the trust-region sphere.

\begin{lemma}[Farther sphere intersection]
\label{lem:farther-intersection}
Let $w\neq0$ and consider the line $x+t w$.  The two solutions of
$\|x+t w\|=r$ are
\begin{equation}
    t_\pm
    =
    \frac{-x^\top w\pm
    \sqrt{(x^\top w)^2+\|w\|^2h^2}}
    {\|w\|^2}.
    \label{eq:sphere-roots}
\end{equation}
If $t_\star$ is the root with the larger absolute value and
$d:=t_\star w$, then
\begin{equation}
    \|d\|
    =
    \frac{|x^\top w|+
    \sqrt{(x^\top w)^2+\|w\|^2h^2}}
    {\|w\|}
    \ge h,
    \label{eq:farther-h}
\end{equation}
and
\begin{equation}
    \|d\|
    \ge
    \frac{2|x^\top w|}{\|w\|}.
    \label{eq:farther-radial}
\end{equation}
In particular, $x+d$ lies on the trust-region sphere.
\end{lemma}

\begin{proof}
Equation~\eqref{eq:sphere-roots} follows by expanding
$\|x+t w\|^2=r^2$ and using $h^2=r^2-\|x\|^2$.
For the farther root,
\[
    |t_\star|
    =
    \frac{|x^\top w|+
    \sqrt{(x^\top w)^2+\|w\|^2h^2}}
    {\|w\|^2}.
\]
Multiplying by $\|w\|$ gives the equality in
\eqref{eq:farther-h}.  The lower bound by $h$ follows from
\[
    \sqrt{(x^\top w)^2+\|w\|^2h^2}\ge\|w\|h,
\]
and \eqref{eq:farther-radial} follows from the bound of the same square root
by $|x^\top w|$.
\end{proof}

We now treat the three regimes used in the main text.

\subsection{Case I: sufficient interior slack}
\label{app:case-slack}

Assume
\begin{equation}
    h\ge\alpha r.
    \label{eq:case1-assump}
\end{equation}
Choose $w=v$.  Let $d$ be the farther-intersection displacement given by
Lemma~\ref{lem:farther-intersection}.  Since $\|v\|=1$,
\eqref{eq:farther-h} gives
\[
    \|d\|\ge h\ge\alpha r.
\]
Moreover, because $d$ is a scalar multiple of $v$,
\begin{equation}
    \frac{d^\top B(x)d}{\|d\|^2}
    =v^\top B(x)v
    \le-\kappa
    \le-\frac{16}{27}\kappa.
    \label{eq:case1-curv}
\end{equation}
Thus, this case satisfies both the displacement and curvature bounds required
in the common conclusion below.

\subsection{Case II: small slack but sufficient radial alignment}
\label{app:case-aligned}

Assume
\begin{equation}
    h<\alpha r,
    \qquad
    |c|\ge\alpha.
    \label{eq:case2-assump}
\end{equation}
Again take $w=v$.  Since $x^\top v=rc$, the farther-intersection estimate
\eqref{eq:farther-radial} yields
\begin{equation}
    \|d\|
    \ge2|x^\top v|
    =2r|c|
    \ge2\alpha r
    \ge\alpha r.
    \label{eq:case2-length}
\end{equation}
The curvature ratio is again exactly the Rayleigh quotient of $v$, and hence
\begin{equation}
    \frac{d^\top B(x)d}{\|d\|^2}
    \le-\kappa
    \le-\frac{16}{27}\kappa.
    \label{eq:case2-curv}
\end{equation}

\subsection{Case III: near-boundary and near-orthogonal geometry}
\label{app:case-hard}

It remains to consider
\begin{equation}
    h<\alpha r,
    \qquad
    |c|<\alpha.
    \label{eq:case3-assump}
\end{equation}
Set
\begin{equation}
    \tau
    :=
    \begin{cases}
        2\alpha\,\operatorname{sgn}(c),&c\neq0,\\
        2\alpha,&c=0,
    \end{cases}
    \qquad
    w:=v+\tau\frac{x}{r}.
    \label{eq:app-w}
\end{equation}
Thus, $|\tau|=2\alpha$.  Because $\|x\|\le r$ and
$\alpha\le1/16$,
\begin{equation}
    \|w\|
    \le1+|\tau|\frac{\|x\|}{r}
    \le1+2\alpha
    \le\frac98.
    \label{eq:app-w-norm}
\end{equation}
We first verify the shifted-curvature bound.  By
Lemma~\ref{lem:shifted-bound},
\begin{align}
    w^\top B(x)w
    &=v^\top B(x)v
      +\frac{2\tau}{r}v^\top B(x)x
      +\frac{\tau^2}{r^2}x^\top B(x)x\notag\\
    &\le
      -\kappa+(2|\tau|+\tau^2)\bar L\notag\\
    &=
      -\kappa+(4\alpha+4\alpha^2)\bar L.
    \label{eq:app-w-curv-raw}
\end{align}
The definition of $\alpha$ gives
$\alpha\bar L\le\kappa/32$ and $\alpha\le1/16$.  Therefore,
\begin{equation}
    (4\alpha+4\alpha^2)\bar L
    \le
    \frac{\kappa}{8}+\frac{\kappa}{128}
    =\frac{17}{128}\kappa,
    \label{eq:app-perturb-bound}
\end{equation}
and hence, in particular,
\begin{equation}
    w^\top B(x)w
    \le-\frac{111}{128}\kappa
    <-\frac34\kappa.
    \label{eq:app-w-curv}
\end{equation}
Combining \eqref{eq:app-w-curv} and \eqref{eq:app-w-norm},
\begin{equation}
    \frac{w^\top B(x)w}{\|w\|^2}
    \le
    -\frac{3\kappa/4}{(9/8)^2}
    =-\frac{16}{27}\kappa.
    \label{eq:app-hard-rayleigh}
\end{equation}

It remains to lower-bound the length of the farther sphere-intersection
step.  Since $h<\alpha r$,
\begin{equation}
    \frac{\|x\|^2}{r^2}
    =1-\frac{h^2}{r^2}
    >1-\alpha^2.
    \label{eq:app-x-near-boundary}
\end{equation}
By the sign choice in \eqref{eq:app-w}, the two terms in
\[
    x^\top w
    =rc+\tau\frac{\|x\|^2}{r}
\]
have the same sign when $c\neq0$, while the first term vanishes when $c=0$.
Consequently,
\begin{equation}
    |x^\top w|
    \ge
    2\alpha\frac{\|x\|^2}{r}
    >2\alpha r(1-\alpha^2).
    \label{eq:app-xw-lower}
\end{equation}
Lemma~\ref{lem:farther-intersection} and
\eqref{eq:app-w-norm} now give
\begin{align}
    \|d\|
    &\ge
    \frac{2|x^\top w|}{\|w\|}\notag\\
    &>
    \frac{4\alpha r(1-\alpha^2)}{9/8}
    \ge\alpha r,
    \label{eq:case3-length}
\end{align}
where the last inequality follows from $\alpha\le1/16$.
Because $d$ is a scalar multiple of $w$,
\eqref{eq:app-hard-rayleigh} yields
\begin{equation}
    \frac{d^\top B(x)d}{\|d\|^2}
    \le-\frac{16}{27}\kappa.
    \label{eq:case3-curv}
\end{equation}

\subsection{Common decrease estimate}
\label{app:common-decrease}

In all three cases, we have constructed a displacement $d$ such that
\begin{equation}
    \|x+d\|=r,
    \qquad
    \|d\|\ge\alpha r,
    \qquad
    \frac{d^\top B(x)d}{\|d\|^2}
    \le-\frac{16}{27}\kappa.
    \label{eq:app-common-properties}
\end{equation}
For interior $x$, Lemma~\ref{lem:perturbation} applies because
$\lambda(x)=0$.  For boundary $x$, it applies because both endpoints lie on
the sphere.  Therefore,
\begin{align}
    q(x+d)-q(x)
    &\le
    e(x)^\top d-\frac8{27}\kappa\|d\|^2
    \le
    S(x)\|d\|-\frac8{27}\kappa\|d\|^2.
    \label{eq:app-master-decrease}
\end{align}
Under the hypothesis of Theorem~\ref{thm:early-escape},
\[
    S(x)
    \le\frac4{27}\kappa\alpha r
    \le\frac4{27}\kappa\|d\|.
\]
Substituting this into \eqref{eq:app-master-decrease} gives
\[
    q(x+d)-q(x)
    \le
    -\frac4{27}\kappa\|d\|^2
    \le
    -\frac4{27}\kappa\alpha^2r^2.
\]
Taking $x^+:=x+d$ proves \eqref{eq:escape-decrease} and completes the proof
of Theorem~\ref{thm:early-escape}.

\end{document}